\documentclass[letterpaper, 10 pt, conference]{IEEEtran}  

\usepackage{url}
\usepackage{array,tabularx}
\usepackage{multicol, nonfloat}  
\usepackage{epsfig} 
\usepackage{amssymb}  
\usepackage{amsthm}
\newtheorem{theorem}{Theorem}[section]

\newtheorem{lemma}[theorem]{Lemma}
\newtheorem{remark}{Remark}

\newtheorem{proposition}[theorem]{Proposition}
\usepackage{blindtext}
\usepackage[utf8]{inputenc}	
\usepackage[english]{babel}
\usepackage{graphicx}
\usepackage{amsmath}
\usepackage{amsfonts}
\usepackage{epstopdf}
\usepackage{cite}
\usepackage{microtype}

\usepackage{color}
\newcommand{\norm}[1]{\| #1 \|}

\newcommand{\quotes}[1]{``#1''}

\title{\LARGE \bf
Primal-Dual Gradient Flow Algorithm for Distributed Support Vector Machines}

\author{Prashant Bansode, Sushant Bahadure, Navdeep Singh
\thanks{Prashant Bansode is with Department of Instrumentation Engineering, Ramrao Adik Institute of Technology, Mumbai, 400706 India.~{\tt\small prashant.bansode@rait.ac.in}}
\thanks{Sushant Bahadure is with Cognizant Technology Solutions, Mumbai, 400076 India.}
\thanks{Navdeep Singh is with department of Electrical Engineering, Veermata Jijabai Technological Institute, Mumbai, 400019 India.}
}
\begin{document}

\maketitle
\thispagestyle{empty}
\pagestyle{empty}

\begin{abstract}
	In this paper, a primal-dual gradient flow algorithm for distributed support vector machines (DSVM) is proposed. A network of computing nodes, each carrying a subset of horizontally partitioned large dataset is considered. The nodes are represented as dynamical systems with Arrow-Hurwicz-Uzawa gradient flow dynamics, derived from the Lagrangian function of the DSVM problem. It is first proved that the nodes are passive dynamical systems. Then, by employing the Krasovskii type candidate Lyapunov functions, it is proved that the computing nodes asymptotically converge to the optimal primal-dual solution. 
\end{abstract}

\section{Introduction}
Support vector machines (SVMs) are supervised learning based paradigms in machine learning, used for classification and regression analysis on raw data\cite{cortes1995support}. Lately, SVMs have gained wide interest in data analytic applications in health-care, power grid, automotive industries etc.,\cite{kumar2010data,nguyen1996classification,terzic2010fluid}. It enables information processing from raw data and help make crucial decisions for the future; using convex optimization based algorithms\cite{boyd2004convex}. However, for applications with huge amount of data, there are often constraints with respect to bandwidth requirement, data storage and processing capability of the computing device, response time, etc. Distributed versions of support vector machines have been proposed as an alternative method to tackle these constraints, as discussed in\cite{stolpe2016distributed,forero2010consensus,wang2012distributed}. The early work of \cite{forero2010consensus} reports the consensus based DSVM technique. In such techniques, the large dataset is partitioned into small datasets and distributed to the computing nodes within the network, wherein each node process the data independent of each other. \cite{stolpe2016distributed} provides an overview of existing distributed support vector machines techniques and proposes a similar technique with horizontally partitioned large datasets. It decomposes the original convex problem into a set of convex-sub problems cast into a distributed alternating method of multipliers framework\cite{boyd2011distributed}, wherein the computing nodes exchange the optimization variables with their neighboring nodes and reach consensus on the optimal solution. This technique makes the algorithm more communication efficient since only the optimization variables are exchanged between computing nodes instead of the support vectors, which may be large in number.

\subsection{Relevant Literature}\label{rl}
In \cite{nedic2009distributed}, authors proposed a distributed subgradient method for optimizing a sum of convex objective functions corresponding to multiagent systems. Recently, the Arrow-Hurwicz-Uzawa gradient flow dynamics based algorithms have become much popular for solving distributed optimization problems\cite{yi2015distributed,simpson2016input,kosaraju2018stability,ding2018primal}. It provides a control-theoretic flavor to the convex optimization problems. Having said that, the optimal solution of the convex optimization problem becomes equivalent to the equilibrium point of such dynamics. \cite{yi2015distributed} proposes this algorithm for distributed convex optimization problem with application to load sharing in power systems. \cite{kosaraju2018stability} integrates Brayton-Moser framework with this approach to solve the constrained convex optimization problem and demonstrates its efficacy by considering an application of building temperature control. Both consider constrained convex optimization problem, and represent the gradient flow dynamics as a switched dynamical system. Here, the asymptotic convergence and stability of the gradient algorithm is proved using properties of hybrid systems theory of\cite{lygeros2003dynamical}. \cite{simpson2016input} considers a convex optimization problem only with equality constraints but gives an exhaustive treatment to these problems. \cite{ding2018primal} utilizes this algorithm for a special case of optimization problems, called distributed resource allocation. 

\subsection{Motivation and Contribution} 
In \cite{kosaraju2018stability}, authors use the passivity property with differentiation of input and output port variables\cite{kosaraju2017control}, to prove the asymptotic convergence of primal-dual dynamics to the optimal solution of the convex optimization problem. This formulation is later applied to a linear support vector machines problem in \cite{kosaraju2018primal}. As discussed already, a single computing machine is inefficient in dealing with SVM algorithm with large datasets. Motivated by the distributed convex optimization techniques discussed in Subsection \ref{rl}, the work presented in this paper intends to develop a \quotes{primal-dual gradient flow algorithm} for distributed support vector machines, much in the spirit of \cite{stolpe2016distributed}. 

The content is organized as follows: Some preliminaries on centralized SVMs and DSMs are presented in Section \ref{prel}.  Section \ref{lf} presents the Lagrangian formulation of the underlying problem and Section \ref{pd} presents the primal-dual gradient flow dynamics of the Lagrangian problem. Section \ref{passive_pd} explores the passivity properties of the dynamics while Section \ref{gb} presents the asymptotic stability of the dynamics. Section \ref{concl} concludes the paper.
\section{Preliminaries}\label{prel}
\subsection{Support Vector Machine problem}
A centralized support vector machine for the case of non-separable data is given below:
\begin{align}
\begin{aligned}
&\min~\frac{1}{2}\norm{w}^2+mC\sum_{i=1}^{n}\xi_{i}\\
&\mathrm{s.t.}~ y_i(w^Tx_i+b)\geq 1-\xi_{i},\forall^n_{i=1},	
\end{aligned}\label{cent}
\end{align}
where $\frac{1}{\norm{w}}$ is the margin that separates positive and negative observations, $(x_i,y_i)\in S$ is a paired observation sample, and $w,b$ are weight and bias variables, respectively. $1-\xi_{i}-y_{i}(w^Tx_{i}+b)$ is called as a hinge loss function. C is used to trade off the sum over all slack variables $\xi$ against the size of the margin. $m>0$ is the scaling factor.
\subsection{Distributed Support Vector Machines}
\subsubsection{Data Partitioning}
It is assumed that the set of observations is horizontally partitioned among a number of nodes in a network, and the underlying optimization problem is solved in a distributed fashion by only exchanging the optimization variables with local nodes\cite{stolpe2016distributed}. 
Consider a network of computing nodes modeled by an undirected graph $G(P,E)$, where vertices $P= \{1,\ldots,m\}$ represent nodes and the set of edges $E$ describes communication links between them. Assuming that the graph is connected and enabling one-hop neighborhood communication, each node $j$ communicates with its neighbors belonging to $\mathcal{N}_j \in P$. Each node $j\in P$ stores a sample of $S_j=\{(x_{j1},y_{j1}),\ldots,(x_{jn},y_{jn})\}$ of labeled observations.
Note that:
\begin{itemize}
	\item $S_j$ is a set of labeled observations allocated to $j^{th}$ machine, $S_j \in S$, where $S$ is a big data set.
	\item $x_{j} \in \Re^{n_j\times 1}$.
	\item $y_{ji} \in \{-1,+1\}$ is a class label.
\end{itemize}
\subsubsection{Convex Optimization Formulation of Distributed Support Vector Machines\cite{stolpe2016distributed}}
The distributed version of the centralized support vector machines \eqref{cent} is given below:
\begin{align}
\begin{aligned}\small
\min~&\frac{1}{2}\sum_{j=1}^{m}\norm{w_j}^2+mC\sum_{j=1}^{m}\sum_{i=1}^{n_j}\xi_{ji}\\
\mathrm{s.t.}&~y_{ji}(w_jx_{ji}+b_j)\geq 1-\xi_{ji},\xi_{ji}\geq 0,\forall j\in P,i=1,\dots,n_j,\\
w_j&=w_l,b_j=b_l,~\forall j\in P,l\in\mathcal{N}_j.
\end{aligned}\label{dist}
\end{align}
The objective function $F:\Re^m\rightarrow \Re$ is continuously differentiable $(C^1)$ and strictly convex. The optimization variables $w,b \in \Re^m$, and $w_j=w_l,b_j=b_l$ are their coupling constraints, where $l$ is a neighbor of $j$ if and only if $l \in \mathcal{N}_j$. Let $h_{ji}(\xi_{ji},w_j,b_j)=1-\xi_{ji}-y_{ji}(w_jx_{ji}+b_j)$.
\section{Problem Formulation}\label{pf}
\subsection{Lagrangian Problem Formulation of Distributed Support Vector Machines}\label{lf}
The Lagrangian function associated with the optimization problem \eqref{dist} is:  
\begin{align}
\begin{aligned}\small 
\mathcal{L}(w,b,\xi,\theta,\mu,\alpha,\beta)=&\frac{1}{2}\norm{w}^2+mC\sum_{j=1}^{m}\sum_{i=1}^{n_j}\xi_{ji}\\&+\alpha^TLw+\beta^TLb\\&+\sum_{j=1}^{m}\sum_{i=1}^{n_j}\theta_{ji}h_{ji}(\xi_{ji},w_j,b_j)\\&+\sum_{j=1}^{m}\sum_{i=1}^{n_j}\mu_{ji}\xi_{ji}+\frac{1}{2}w^TLw+\frac{1}{2}b^TLb
\end{aligned}\label{lagr}
\end{align}
where $\theta_{ji},\mu_{ji}$ are the Lagrange multipliers associated with inequalities $h_{ji}(\xi_{ji},w_j,b_j)$ and $\xi_{ji}\geq 0$, of $j^{th}$ computing node, and $\alpha_j,\beta_j$ are the Lagrange multipliers associated with coupling constraints of $j^{th}$ and $l^{th}, \forall l \in \mathcal{N}_j$ nodes. $L$ is the Laplacian matrix of the undirected graph $G$. The corresponding Lagrangian dual problem of \eqref{dist} is stated as follows:
\begin{align}
\begin{aligned}
&\max_{\alpha,\beta,\theta,\mu}~\min_{w,b,\xi} \mathcal{L}(w,b,\theta,\alpha,\beta)\\
&\mathrm{s.t.}~\theta_{ji} \geq 0,\mu_{ji} \geq 0,\alpha_j\geq 0,\beta_j\geq 0, \forall i = 1,\ldots, n_j,\forall j \in P.
\end{aligned}\label{dual}
\end{align} 
\begin{remark}
Since the primal problem \eqref{dist} is convex, it is assumed that strong duality holds for \eqref{dist}-\eqref{dual} and Slater's condition is satisfied. The pair $(w^*,b^*)$ is an optimal solution to \eqref{dist} if there exist $(\theta^*, \alpha^*,\beta^*)$ such that the following Karush-Kuhn-Tucker (KKT) conditions are satisfied. 
\end{remark}
\begin{align}\small
\begin{aligned}
&\nabla \mathcal{L}(w^*,b^*,\xi^*_{ji},\mu^*_{ji},\theta^*,\alpha^*,\beta^*)=\mathbf{0},\\
&h_{ji}(\xi^*_{ji},w^*_j,b^*_j)\leq 0,\forall j\in P,i=1,\dots,n_j,\\
&\alpha^*_j,\beta^*_j,\theta^*_{ji} \geq 0,\mu^*_{ji} \geq 0,\forall j\in P,i=1,\dots,n_j,\\
&\theta^*_{ji}h_{ji}(\xi^*_{ji},w^*_j,b^*_j)=0,\xi^*_{ji}\mu^*_{ji}=0,\\
&w^*_j = w^*_l,b^*_j = b^*_l,\forall j\in P,l\in\mathcal{N}_j.\\
\end{aligned}\label{kkt}
\end{align}
\subsection{Primal-dual dynamics}\label{pd}
The Arrow-Hurwicz-Uzawa gradient flow dynamics for the Lagrangian \eqref{lagr} are derived as follows:
\begin{align}
primal:\begin{cases}
\begin{aligned}\small
\dot{w}_j&=-\nabla_{w_j} \mathcal{L}(w,b,\xi,\theta,\mu,\alpha,\beta),\\
\dot{b}_j&=-\nabla_{b_j} \mathcal{L}(w,b,\xi,\theta,\mu,\alpha,\beta),\\
\dot{\xi}_{ji}&=-\nabla_{\xi_{ji}} [\mathcal{L}(w,b,\xi,\theta,\mu,\alpha,\beta)]^+_{\xi_{ji}},\\
\end{aligned}\label{pdd1}
\end{cases}
\end{align}
\begin{align}
dual:\begin{cases}
\begin{aligned}
\dot{\alpha}_j&=\nabla_{\alpha_j} \mathcal{L}(w,b,\xi,\theta,\mu,\alpha,\beta),\\
\dot{\beta}_j&=\nabla_{\beta_j} \mathcal{L}(w,b,\xi,\theta,\mu,\alpha,\beta),\\
\dot{\theta}_{ji}&=\nabla_{\theta_{ji}} [\mathcal{L}(w,b,\xi,\theta,\mu,\alpha,\beta)]^+_{\theta_{ji}},\\
\dot{\mu}_{ji}&=\nabla_{\mu_{ji}} [\mathcal{L}(w,b,\xi,\theta,\mu,\alpha,\beta)]^+_{\mu_{ji}}\label{pdd2}
\end{aligned}
\end{cases}
\end{align}
Corresponding to \eqref{pdd1}-\eqref{pdd2}, the primal-dual dynamics for $j^{th}$ node in the network is given as follows:
\begin{align}\small
primal:
\begin{aligned}
\begin{cases}
\dot{w}_j=&-w_j-\sum_{i=1}^{n_j}\theta_{ji}(-y_{ji}x_{ji})\\
&-\sum_{l\in \mathcal{N}_j}^{}(\alpha_j-\alpha_l)-\sum_{l\in \mathcal{N}_j}(w_j-w_l),\\
\dot{b}_j=&-\sum_{i=1}^{n_j}\theta_{ji}(-y_{ji})-\sum_{l\in \mathcal{N}_j}^{}(\beta_j-\beta_l)\\
&-\sum_{j \in \mathcal{N}_j}^{}(b_j-b_l),\\
\dot{\xi}_{ji}=&[-mC-\mu_{ji}+\theta_{ji}]^+_{\xi_{ji}}
\end{cases}
\end{aligned}\label{dyn11}
\end{align}
\begin{align}
\begin{aligned}dual:
\begin{cases}
\dot{\theta}_{ji}=&[h_{ji}(\xi_{ji},w_j,b_j)]^+_{\theta_{ji}},\\
\dot{\mu}_{ji}=&[\xi_{ji}]^+_{\mu_{ji}},\\
\dot{\alpha}_j =& \sum_{l\in \mathcal{N}_j}(w_j-w_l),\\
\dot{\beta}_j =& \sum_{l\in \mathcal{N}_j}(b_j-b_l).
\end{cases}
\end{aligned}\label{dyn12}
\end{align}
\subsubsection{Partitioned Primal-dual dynamics}
\begin{figure}[t]
	\centering
	\includegraphics[width=2.65in]{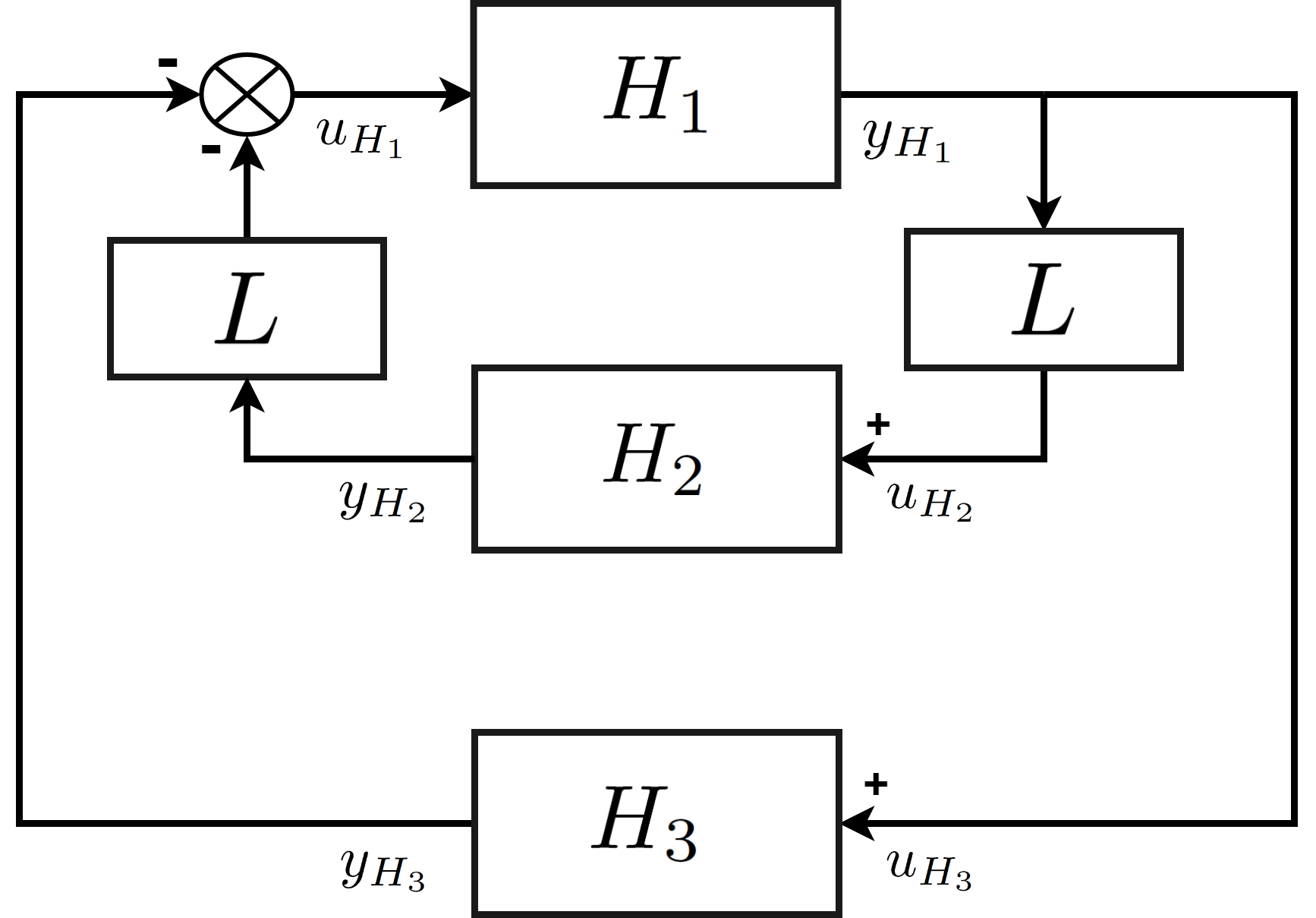}
	\caption{Interconnected structure of $H_1, H_2$ and $H_3$ representing the network dynamics.}
	\label{interconnect}
\end{figure}
In the following, the network representing the dynamical systems of the form \eqref{dyn11}-\eqref{dyn12}, is partition into three subsystems as $H_1$, $H_2$, and $H_3$ as shown in Fig. \ref{interconnect}.
The subsystem $H_1$ contains only primal variables, with $u_{H_1}$ and $y_{H_1}$ as its input and output respectively, as given below:
\begin{align}\small
H_1&:\begin{cases}
\dot{w}&=-w- Lw-L\alpha-\zeta,\\
\dot{b}&=- Lb-L\beta-\eta,\\
u_{H_1}&=-[(L\alpha)^T,(L\beta)^T,\zeta^T,\eta^T]^T,\\
y_{H_1}&=[w^T,b^T]^T.
\end{cases}\label{h1}
\end{align}
The subsystem $H_2$ contains only consensus-dual variables, with $u_{H_2}$ and $y_{H_2}$ as its input and output respectively, as given below:
\begin{align}
H_2&:\begin{cases}
\dot{\alpha}&=Lw,\\
\dot{\beta}&=Lb,\\
u_{H_2}&=[(Lw)^T,(Lb)^T]^T,\\
y_{H_2}&=[\alpha^T,\beta^T]^T.
\end{cases}\label{h2}
\end{align}
The subsystem $H_3$ contains the slack variable, and the dual variables corresponding to the inequality constraints, with $u_{H_3}$ and $y_{H_3}$ as its input and output respectively, as given below:
\begin{align}
H_3&:\begin{cases}
\dot{\theta}_{ji}=&[h_{ji}(\xi_{ji},w_j,b_j)]^+_{\theta_{ji}}\forall^{n_j}_{i=1},\forall^m_{j=1},\\
\dot{\mu}_{ji}=&[\xi_{ji}]^+_{\mu_{ji}}\forall^{n_j}_{i=1},\forall^m_{j=1},\\
\dot{\xi}_{ji}=&[-mC-\mu_{ji}+\theta_{ji}]^+_{\xi_{ji}}\forall^{n_j}_{i=1},\forall^m_{j=1},\\
u_{H_3} &= [w^T,b^T]^T,\\
y_{H_3} &= [\zeta^T,\eta^T]^T.
\end{cases}\label{h3}
\end{align}
where $\zeta, \eta, \mu\in \Re^m$, and $\zeta_j = \sum_{i=1}^{n_j}\theta_{ji}(-y_{ji}x_{ji})$ with $\eta_j = \sum_{i=1}^{n_j}\theta_{ji}(-y_{ji})$.
\begin{remark}
It is seen from Fig. \ref{interconnect} that the subsystems $H_1$ and $H_2$ are feedback interconnected while the subsystem $H_3$ is feedback interconnected with the feedback structure of $H_1$ and $H_2$. Fig. \ref{interconnect} represents the entire network dynamics of distributed support vector machines.
\end{remark}
\subsection{Passivity based stability analysis of primal-dual dynamics}\label{passive_pd}
In what follows, the passivity properties of subsystems $H_1,H_2$, and $H_3$ are explored. Later, these properties will be used to prove that the network of dynamical systems represented in Fig. \ref{interconnect} is asymptotically convergent and stable. 

The passivity based stability analysis of these subsystems is presented in the subsequent sections.
\subsubsection{System $H_1$ is passive}
\begin{proposition}
	Assuming that the graph $G$ is connected and $f(w)$ is strictly convex, if there exist a pair $(w_{eq},b_{eq})$ that satisfy \eqref{kkt}, then the subsystem $H_1$ is passive with port variables $(\dot{\tilde{y}}_{H_1},\dot{u}_{H_1})$. Further, each solution of the dynamics in subsystem $H_1$ asymptotically converges to $(w_{eq},b_{eq})$ for ${u}_{H_1} = \mathbf{0}$.
\end{proposition}
\begin{proof}
Consider a Krasovskii-Lyapunov storage function for $H_1$ as given below:
\begin{align}
\begin{aligned}
V_{H_1}(w,b) &= \frac{1}{2}\dot{w}^T\dot{w}+\frac{1}{2}\dot{b}^T\dot{b}.
\end{aligned}\label{vh1}
\end{align}
Differentiating \eqref{vh1} w.r.t. time yields,
\begin{align}
\begin{aligned}\small 
\dot{V}_{H_1}(w,b)
=&\dot{w}^T\ddot{w}+\dot{b}^T\ddot{b}  \\
=&\dot{w}^T \{-\dot{w}-L\dot{w}-L\dot{\alpha}-\dot{\zeta}\}  \\ 
&+ \dot{b}^T \{-L\dot{b}-L\dot{\beta}-\dot{\eta}\},\\
=&-\dot{w}^T\dot{w}-\dot{w}^TL\dot{w}-\dot{w}^TL\dot{\alpha}-\dot{w}^T\dot{\zeta}  \\
&-\dot{b}^TL\dot{b}-\dot{b}^TL\dot{\beta}-\dot{b}^T\dot{\eta}.
\end{aligned}\label{dh1}
\end{align}
In the following, \eqref{dh1} is used to provide a ISS-Lyapunov like inequality for $\dot{V}_{H_1}$.   
\begin{align}
\dot{V}_{H_1}(w,b)&=-\dot{w}^T\dot{w}-\dot{w}^TL\dot{w}-\dot{b}^TL\dot{b}+\dot{\tilde{y}}_{H_1}^T\dot{u}_{H_1},\label{dh11}\\
&\leq-\dot{w}^T\dot{w}-\dot{w}^TL\dot{w}-\dot{b}^TL\dot{b}+\frac{\gamma_{H_1}}{4}\norm{\dot{\tilde{y}}_{H_1}}^2\nonumber\\&~~~+\frac{1}{\gamma_{H_1}}\norm{\dot{u}_{H_1}}^2,\nonumber\\
&\leq-(1+\lambda_2(L)-\frac{\gamma_{H_1}}{2})\norm{\dot{w}}^2-(\lambda_2(L)-\frac{\gamma_{H_1}}{2})\norm{\dot{b}}^2\nonumber\\&~~~+\frac{1}{\gamma_{H_1}}\norm{\dot{u}_{H_1}}^2 \label{dh22}
\end{align}
where $\tilde{y}_{H_1} = [y^T_{H_1},y^T_{H_1}]^T$. Further the $L_2$-gain of the system $H_1$, with both input and output measured in terms of the $L_2$ norm, is computed as follows:
Reconsider \eqref{dh22}, 
\begin{align}
\dot{V}_{H_1}(w,b)&\leq-\psi({V}_{H_1}(w,b))+\frac{1}{\gamma_{H_1}}\norm{\dot{u}_{H_1}}^2,\label{passh1}
\end{align} 
where, $\psi({V}_{H_1}(w,b))$ is a positive definite function of ${V}_{H_1}(w,b)$, and $\frac{1}{\gamma_{H_1}}$ is the $L_2$-gain of $H_1$ from the port input $\dot{u}_{H_1}$ to the port output $\dot{y}_{H_1}$. Further, $H_1$ is $L_2$ stable if the induced gain of $H_1$ is less than or equal to $\frac{1}{\gamma_{H_1}}$, with $\gamma_{H_1}\leq 2 \lambda_2(L)$, where $\lambda_2(L)>0$ as long as the graph $G$ is connected.

The inequality \eqref{passh1} implies that the subsystem $H_1$ is \quotes{output strictly passive (OSP)} with respect to the port variables $\dot{u}_{H_1}$ and $\dot{y}_{H_1}$. Secondly, from \eqref{dh11}, the primal solutions asymptotically converge to $(w_{eq},b_{eq})$, when $\dot{V}_{H_1}(w,b)<0$, for ${u}_{H_1} = \mathbf{0}$.
\end{proof}	
\subsubsection{System $H_2$ is passive}
\begin{proposition}
Assuming that the graph $G$ is connected and $f(w)$ is strictly convex, if there exists a pair $(\alpha_{eq},\beta_{eq})$ satisfying \eqref{kkt}, then the subsystem $H_2$ is passive with port variables $(\dot{y}_{H_2},\dot{u}_{H_2})$.
\end{proposition}
\begin{proof}
Consider a Krasovskii-type storage function\cite{feijer2010stability} as given below:
\begin{align}
\dot{V}_{H_2}(\alpha,\beta) = \frac{1}{2}\dot{\alpha}^T\dot{\alpha}+\frac{1}{2}\dot{\beta}^T\dot{\beta}.\label{vh2}
\end{align}
Differentiating \eqref{vh2} w.r.t. time yields,
\begin{eqnarray}
\dot{V}_{H_2}(\alpha,\beta)
&=& \dot{\alpha}^T\ddot{\alpha}+\dot{\beta}^T\ddot{\beta}, \nonumber \\
&=& \dot{\alpha}^TL\dot{w}+\dot{\beta}^TL\dot{b}, \label{dh2}\\
&=&\dot{y}_{H_2}^T\dot{u}_{H_2}. \label{dh21}
\end{eqnarray}
Similar to \eqref{passh1}, the following can be defined,
\begin{align}
V_{H_2}(\alpha(\tau),\beta(\tau))-V_{H_2}(\alpha(0),\beta(0)) \leq \int_{0}^{\tau} \dot{y}_{H_2}^T\dot{u}_{H_2}dt.\label{passh2}
\end{align}
Hence, the subsystem $H_2$ is passive with respect to port variables $\dot{u}_{H_2}$ and $\dot{y}_{H_2}$. The dual solutions asymptotically converge to $(\alpha_{eq},\beta_{eq})$ when ${u}_{H_2}=\mathbf{0}$.
\end{proof}
\subsubsection{System $H_3$ is passive}
Let us proceed first with the dual variable $\theta_{ji}$, its dynamics can also be written as:
\begin{align}
\small \dot{\theta}_{ji} &=
\begin{cases}\small 
h_{ji}(w_j,b_j)~\textrm{if}~\theta_{ji}>0,\forall i,\\
\max\{0,h_{ji}(w_j,b_j)\}~\textrm{if}~\theta_{ji}=0.
\end{cases}\label{ineq}
\end{align}
\eqref{ineq} becomes discontinuous when $\theta_{ji}=0$ and $ h_{ji}(w_j,b_j)<0$. The value of $h_{ji}(w_j,b_j)^+$ switches from $h_{ji}(w_j,b_j)$ to $0$. To further clarify that, \eqref{ineq} is reformulated as given below.
\begin{align}
\small \theta_{ji} =\begin{cases}
h_{ji}(w_j,b_j)^+=h_{ji}(w_j,b_j),~\mathrm{if}~\theta_{ji}>0~\mathrm{or}~ h_{ji}(w_j,b_j)>0\\
h_{ji}(w_j,b_j)^+=0. 
\end{cases}\label{proj}
\end{align}
From \eqref{proj}, the projection is seen to be active for the second case. Let $\mathcal{I}_j= \{1,\ldots,n_j\}$ and $\sigma_{j}:[0,\infty) \rightarrow \mathcal{I}_j,\forall i \in \mathcal{I}_j,\forall j \in P$ be a switching signal. Then
\begin{align}\small
\sigma_{j}(t)=\{i|\theta_{ji}=0,h_{ji}(w_j,b_j)\leq 0,\forall i \in \mathcal{I}_j,\forall j \in P\}\label{switch}
\end{align} is valid for the active projection.
Considering \eqref{switch}, the inequality constraint dynamics given in \eqref{h3} takes the form of a switched dynamical system\cite{kosaraju2018stability}, as follows:
\begin{align}
\dot{\theta}_{ji}=h_{ji}(w_j,b_j,\sigma_{ji})=\begin{cases}
h_{ji}(w_j,b_j),\forall i \notin \sigma_{j}(t),\\
0,\forall i \in \sigma_{j}(t).  
\end{cases}\label{thetadot}
\end{align}
Notice that the dynamics of the dual variable $\mu_{ji}$ and the slack variable $\xi_{ji}$ can also be represented as the switched dynamical system:
\begin{align}
\dot{\mu}_{ji}=\begin{cases}
\xi_{ji},\forall i \notin \iota_{j}(t),\\
0,\forall i \in \iota_{j}(t),  
\end{cases}\label{mudot}
\end{align}
and,
\begin{align}
\dot{\xi}_{ji}=\begin{cases}
-mC-\mu_{ji}+\theta_{ji},\forall i \notin \rho_{j}(t),\\
0,\forall i \in \rho_{j}(t).  
\end{cases}\label{xidot}
\end{align}

where $\iota_{j}(t),\rho_{j}(t)$ are the switching signals corresponding to \eqref{mudot} and \eqref{xidot}, respectively.
\begin{proposition}
Let $\theta_{{ji}(eq)},\xi_{ji(eq)},\mu_{ji(eq)}\forall i,\forall j$ satisfy \eqref{kkt}, and $V_{H_3}(\theta,\mu,\xi,\sigma,\iota,\rho)$ be the Krasovskii-Lyapunov storage function associated with the system $H_3$; then the subsystem $H_3$ is passive with port input $\dot{u}_{H_3}$, and port output $\dot{y}_{H_3}$ for:
\begin{enumerate}
	\item each pair of switching instances $(\tau^+_\sigma,\tau^-_\sigma)$ corresponding to \eqref{thetadot}.
	\item each pair of switching instances $(\tau^+_\mu,\tau^-_\mu)$ corresponding to \eqref{mudot}.
	\item each pair of switching instances $(\tau^+_\rho,\tau^-_\rho)$ corresponding to \eqref{xidot}.
\end{enumerate} 
\end{proposition}
\begin{proof}
Let $V_{H_3}$ be defined as given below:
\begin{align}
V_{H_3} = \frac{1}{2}\sum_{j=1}^{m}\Bigg\{\sum_{i \notin \sigma_j}\dot{\theta}^2_{ji}+\sum_{i \notin \iota_j}\dot{\mu}^2_{ji}+\sum_{i \notin \rho_j}\dot{\xi}^2_{ji}\Bigg\}, \forall j\in P.\label{vh3}
\end{align}
Differentiating \eqref{vh3} with respect to time yields, 
\begin{align}
\dot{V}_{H_3}&=\sum_{j=1}^{m}\Bigg\{\sum_{i \notin \sigma_j}^{}\dot{\theta}_{ji}\ddot{\theta}_{ji}+\sum_{i \notin \rho_j}^{}\dot{\xi}_{ji}\ddot{\xi}_{ji}+\sum_{i \notin \iota_j}^{}\dot{\mu}_{ji}\ddot{\mu}_{ji}\Bigg \}\nonumber\\
&=\sum_{j=1}^{m}\Bigg\{\sum_{i \notin \sigma_j}^{}\dot{\theta}_{ji}[\nabla_{(w_j,b_j)}h_{ji}(w_j,b_j)]^T\small \left [{\begin{array}{c}
	\dot{\xi}_{ji}\\
	\dot{w}_j\\
	\dot{b}_j
	\end{array}}\right]\nonumber\\
&~~+\sum_{i \notin \rho_j}^{}\dot{\xi}_{ji}[\nabla_{(\mu_{ji},\theta_{ji})}\dot{\xi}_{ji}]^T\small \left [{\begin{array}{c}
	\dot{\mu}_{ji} \\
	\dot{\theta}_{ji}
	\end{array}}\right]\\
&~~+\sum_{i \notin \rho_j}^{}\dot{\mu}_{ji}\dot{\xi}_{ji}\Bigg \}\\
&\leq\sum_{j=1}^{m}-\dot{\xi}_j\dot{\theta}_j+\dot{\zeta}_j\dot{w}_j+\dot{\eta}_j\dot{b}_j-\dot{\xi}_j\dot{\mu}_j+\dot{\xi}_j\dot{\theta}_j+\dot{\xi}_j\dot{\mu}_j \label{hps}\\
&\leq\dot{w}^T\dot{\zeta}+\dot{b}^T\dot{\eta}\label{dh31-1}\\
&\leq\dot{y}_{H_3}^T\dot{u}_{H_3}.\label{dh31}
\end{align}
The inequality \eqref{dh31-1} clearly indicates that $\dot{V}_{H_3}$ does not depend on the variables $\xi, \mu$. Hence, the inequality \eqref{dh31} can be written only in terms of the variable $\theta_{ji}$ as:
\begin{align}
\sum_{j=1}^{m}\sum_{i \notin \sigma_j}^{}V_j(\theta_{ji}(\tau^+_\sigma))-V_j(\theta_{ji}(\tau^-_\sigma))\leq\int_{\tau^-_\sigma}^{\tau^+_\sigma}\dot{y}_{H_3}^T\dot{u}_{H_3}dt. \label{passivity}
\end{align} 
\eqref{passivity} ensures that the switched system \eqref{thetadot} represents a finite family of passive systems. However, it must be ensured that the Lyapunov function $V_{H_3}$ does not increase during the switching events. In line with this, the following two cases have been considered:
\begin{enumerate}
	\item It may happen for some $j \in P$ and corresponding constraint $i$ in \eqref{thetadot}, that the function $h_{ji}(w_j,b_j)$ goes from negative to positive through $0$. This will cause the Lyapunov function to change from $V_j(\theta_{ji}(\tau_\sigma^-))$ to $V_j(\theta_{ji}(\tau_\sigma^+))$. If that happens, the Lagrangian multiplier $\theta_{ji}>0$ will add a new term to $V_j(\theta_{ji}(\tau_\sigma))$. Since, $V_j(\theta_{ji}(\tau_\sigma))$ is continuous in time, \eqref{passivity} holds for $\tau <\tau_\sigma^-$ as well as $\tau>\tau_\sigma^+$. Hence, 
	\begin{align}
	V_j(\theta_{ji}(\tau_\sigma^+))=V_j(\theta_{ji}(\tau_\sigma^-)).
	\end{align}
	\item In this case the projection of $i^{th}$ constraint becomes active, i.e., $\theta_{ji}$ reaches to $0$ from a positive value for the $i^{th}$ constraint of the $j^{th}$ machine. Hence, the corresponding $i^{th}$ term of the Lyapunov function $V_j(\theta_{ji})$ will disappear since $i \in  \sigma(t)$. In turn, the following inequality will be satisfied.
	\begin{align}
	V_j(\theta_{ji}(\tau_\sigma^+))<V_j(\theta_{ji}(\tau_\sigma^-)).
	\end{align}
\end{enumerate}
Hence, in both the cases, the Lyapunov function $V_j(\theta_{ji}(\tau_\sigma))$ will be non-increasing. 
\end{proof}
\subsection{Stability analysis of the feedback interconnection shown in Fig \ref{interconnect}.}\label{gb}	
\begin{proposition}\label{prof3.4}
	The feedback interconnection of the subsystems $H_1,H_2$, and $H_3$ is passive and asymptotically stable.
\end{proposition}
\begin{proof}
	Let $V$ be the candidate Lyapunov function for the interconnected system represented in Fig. \ref{interconnect}, as given below:
\begin{align}
V = V_{H_1}+V_{H_2}+V_{H_3}. \label{ly}
\end{align}
Then,
\begin{align}\small
\dot{V} &= \dot{V}_{H_1}+\dot{V}_{H_2}+\dot{V}_{H_3},\nonumber\\
&=-\dot{w}^T\dot{w}-\dot{w}^TL\dot{w}-\dot{b}^TL\dot{b}\nonumber\\
&\underbrace{-\dot{\tilde{y}}_{H_1}^T\dot{u}_{H_1}+\dot{y}_{H_2}^T\dot{u}_{H_2}+\dot{y}_{H_3}^T\dot{u}_{H_3}},\nonumber\\
&\leq -(1+\lambda_2(L))\norm{\dot{w}}^2-\lambda_2(L)\norm{\dot{b}}^2+{0}, \nonumber\\
&\leq 0.\label{hwz}
\end{align}
Since $L$ is a Laplacian matrix of the connected graph $G$, $\lambda_2(L)>0$ always holds. Hence $\dot{V}\leq 0$. 
\begin{remark}
	The expressions \eqref{passh1}, \eqref{passh2}, \eqref{passivity}, and \eqref{hwz} prove that the feedback interconnection of the subsystems $H_1,H_2$, and $H_3$ is passive\cite{khalil1996noninear}. Note that from \eqref{h1}-\eqref{h3}, 
	\begin{align}
	\underbrace{\dot{\tilde{y}}_{H_1}^T\dot{u}_{H_1}+\dot{y}_{H_2}^T\dot{u}_{H_2}+\dot{y}_{H_3}^T\dot{u}_{H_3}}&=\dot{\tilde{y}}_{H_1}^T\dot{u}_{H_1}-\dot{\tilde{y}}_{H_1}^T\dot{u}_{H_1}\nonumber\\
	&=0.
	\end{align}
	\eqref{hwz} reveals an interesting fact that the asymptotic convergence and stability of the proposed algorithm relies only on the primal dynamics. 
\end{remark}

The following result helps to establish the boundedness of the trajectories of \eqref{dyn11}-\eqref{dyn12}.
\begin{lemma}
The trajectories of \eqref{dyn11}-\eqref{dyn12} are bounded for any finite initial conditions.
\end{lemma}
\begin{proof}
    The proof of this Lemma is given in \cite{yi2015distributed,kosaraju2018stability}, hence omitted from this work.
\end{proof}

The extension of LaSalle invariance principle for hybrid dynamical systems \cite{lygeros2003dynamical}, is stated below, which in our case provides a useful result on the convergence of primal-dual dynamics \eqref{h1}-\eqref{h3} to the solution of optimal solution that satisfies \eqref{kkt}. Without loss of generality, the network dynamics \eqref{h1}-\eqref{h3}, is now considered as a hybrid system.
\begin{lemma}\label{hyb}
Consider the hybrid dynamical system \eqref{h1}-\eqref{h3}. Let $\Psi$ be a compact, positively invariant set. Assuming that the Lyapunov function $V$ defined in \eqref{ly} is continuously differentiable and $\dot{V} \leq 0$ along the trajectories of $(w(t),b(t),\xi(t),\theta(t),\mu(t),\alpha(t),\beta(t)) \in \Psi$, every trajectory in $\Psi$ converges to $\epsilon$, where $\epsilon \subset \Psi$ is a maximal positive invariant set of $\Psi$ such that
\begin{enumerate}
    \item $\dot{V}=0$ for a fixed $\sigma$.
    \item $V_j(\theta_{ji}(\tau_\sigma^+))=V_j(\theta_{ji}(\tau_\sigma^-))$ for a switching instance $t$ between $\sigma^-$ and $\sigma^+$.
\end{enumerate}
\end{lemma}
\end{proof}
Lemma \ref{hyb} gives the next result on the convergence of primal-dual dynamics \eqref{dyn11}-\eqref{dyn12} to the optimal solutions that satisfies the conditions in \eqref{kkt}.
\begin{theorem}\label{thmn-1}
The hybrid dynamical system \eqref{h1}-\eqref{h3} converges to the optimal solutions $w^*,b^*,\xi^*,\theta^*,\mu^*,\alpha^*,\beta^*$ satisfying \eqref{kkt}.
\end{theorem}
\begin{proof}
From Lemma \ref{hyb}, for a fixed $\sigma$, $\dot{V}=0$. Thus the primal dynamics in \eqref{h1} converges to the optimal solution of \eqref{dist} when the trajectories of $w_j(t),b_j(t),\forall j \in P$ reach the set $\epsilon$ and converge to, i.e. $w_j=w_l=w^*,b_j=b_l=b^*,\forall j,l \in P$. Simultaneously, the dual variables also reach consensus, i.e. $\alpha_j=\alpha_l=\alpha^*,\beta_j=\beta_l=\beta^*,\forall j,l \in P$. If $h_{ji}(w^*_j,b^*_j,\xi^*_{ji})<0$ then $\theta^*_{ji}=0$. However, if $h_{ji}(w^*_j,b^*_j)>0$, then $\theta^*_{ji}$ will penalize the constraint violation by approaching a very large value. Since all trajectories are bounded, it contradicts the continuity of $V$, thus $\dot{\theta}_{ji}=0,\forall j \in P, \forall i = 1, \ldots, n_j$. This holds for $\mu_{ji}$ too. Thus slack variable $\xi_{ji}$ approaches $\xi^*_{ji}$ as $\theta_{ji}$ and $\mu_{ji}$ approach $\theta^*_{ji}$ and $\mu^*_{ji}$, respectively. Thus, $[\dot{w}^T,\dot{b}^T,\dot{\alpha}^T,\dot{\beta}^T]^T=\mathbf{0}$, where $\mathbf{0}$ is a zero vector of the appropriate dimensions. To this end, the fixed point solutions of \eqref{h1}-\eqref{h3} also satisfy the KKT conditions \eqref{kkt} and yield the optimal solutions of \eqref{dist} and \eqref{dual}.
\end{proof}
\begin{theorem}\label{thmn}
The optimal solution of \eqref{lagr} is asymptotically stable. 
\end{theorem}
\begin{proof}
The proof is straightforward from Proposition \ref{prof3.4} and Theorem \ref{thmn-1}.
\end{proof}

\begin{remark}
	From inequalities \eqref{passh1}, \eqref{passh2}, and \eqref{passivity}, it is apparent that the network dynamics comprising \eqref{h1}-\eqref{h3} is passive, and inherently robust to perturbations arising in primal and dual variables [see, Proposition 4.3.1, Remark 4.3.3 \cite{van2000l2}].  
\end{remark}
\section{Conclusions}\label{concl}
The paper primarily focuses on continuous-time primal-dual gradient flow algorithm for the distributed support vector machines with the case of horizontally partitioned large dataset. It is proved that the algorithm is passive and asymptotically convergent. Using hybrid LaSalle invariance principle, it is proved that the optimal solution is asymptotically stable.
\bibliographystyle{unsrt}
\bibliography{dsvm_pd} 
\end{document}